\documentclass[11pt,a4paper]{article}

\usepackage[french,english]{babel}
\usepackage[T1]{fontenc}
\usepackage[dvips]{graphicx}
\usepackage{amsmath,amsfonts,amssymb,amsthm,bbm,latexsym,graphicx,color}
\usepackage{amsmath,amssymb,exscale,epsfig,psfrag}

\numberwithin{equation}{section}

\newtheorem{theoreme}{Theorem}
\newtheorem{lemme}{Lemma}[section]
\newtheorem{proposition}[lemme]{Proposition}

\newtheorem{definition}[lemme]{Definition}
\title{{\bf Continuum Percolation for Quermass Model  
}}

\author{
David {\sc Coupier}\\
{\footnotesize Laboratoire Paul Painlevé U.M.R. CNRS 8524}\\
{\footnotesize  Université Lille 1, France}\\ 
{\footnotesize e-mail~:~david.coupier@math.univ-lille1.fr}\\[7mm]
David {\sc Dereudre} \\
{\footnotesize Laboratoire Paul Painlevé U.M.R. CNRS 8524}\\
{\footnotesize  Université Lille 1, France}\\
{\footnotesize e-mail~:~david.dereudre@math.univ-lille1.fr}}

\newcommand{\E}{\ensuremath{{\mathcal{E}}}}

\newcommand{\B}{\ensuremath{{\mathcal{B}}}}
\newcommand{\A}{\ensuremath{{\mathcal{A}}}}

\renewcommand{\S}{\ensuremath{{\mathcal{S}}}}
\newcommand{\F}{\ensuremath{{\mathcal{F}}}}
\newcommand{\LL}{\ensuremath{{\mathcal{L}}}}

\newcommand{\R}{\ensuremath{{\mathbb{R}}}}

\newcommand{\Rd}{\ensuremath{{\mathbb{R}^2}}}

\newcommand{\NN}{\ensuremath{{\mathbb{N}}}}

\renewcommand{\t}{\ensuremath{\theta}}
\newcommand{\x}{\ensuremath{\chi}}

\newcommand{\1}{\ensuremath{\mbox{\rm 1\kern-0.23em I}}}
\renewcommand{\L}{\ensuremath{\Lambda}}
\renewcommand{\l}{\ensuremath{\lambda}}
\renewcommand{\o}{\ensuremath{\omega}}
\renewcommand{\O}{\ensuremath{\Omega}}
\newcommand{\eps}{\ensuremath{\varepsilon}}

\newcommand{\Z}{\ensuremath{{\mathbb{Z}}}}

\begin{document}
\maketitle

\centerline{\textbf{Abstract}} 
The continuum percolation for Markov (or Gibbs) germ-grain models is investigated. The grains are assumed circular with random radii on a compact support. The morphological interaction is the so-called Quermass interaction defined by a linear combination of the classical Minkowski functionals (area, perimeter and Euler-Poincaré characteristic). We show that the percolation occurs for any coefficient of this linear combination and for a large enough activity parameter. An application to the phase transition of the multi-type Quermass model is given.


\vspace{5mm}
\noindent
KEY-WORDS: Stochastic geometry, Gibbs point process, germ-grain model, Quermass interaction, percolation, phase transition.
\newpage

\section{Introduction}

The \textit{germ-grain model} is built by unifying random convex sets-- \textit{the grains} --centered at the points-- \textit{the germs} --of a spatial point process. It is used for modelling random surfaces and interfaces, geometrical structures growing from germs, etc. For such models, the \textit{continuum percolation} refers mainly to the existence of an unbounded connected component. This phenomenon expresses some macroscopic properties of materials as permeability, conductivity, etc. Moreover, it turns out to be an efficient tool to exhibit phase transition in Statistical Mechanics \cite{CCK,GH}. For theses reasons, the continuum percolation has been abundantly studied since the eighties and the pioneer paper of Hall \cite{Hall}.

When the grains are independent and identically distributed, and the germs are given by the locations of a Poisson point process (PPP), the germ-grain model is known as the \textit{Boolean model}. In this context, the continuum percolation is well understood; see the book of Meester and Roy \cite{MR} for a very complete reference. One of the first results is the existence of a percolation threshold $z^*$ for the intensity parameter $z$ of the stationary PPP: provided the mean volume of the grain is finite, percolation occurs for $z>z^*$ and not for $z<z^*$.

Because of the independence properties of the PPP, the Boolean model is sometimes caricatural for the applications in Biology or Physics. Mecke and its coauthors \cite{LMW,M} have mentioned the need of developing, via Markov or Gibbs process, an interacting germ-grain model in which the interaction would locally depend on the geometry of the set. For this purpose, let us cite the Widom-Rowlinson model \cite{WR}, the area interaction process \cite{BV} and the morphological model \cite{M}. Thus, Kendall, Van Lieshout and Baddeley suggested in \cite{KVB} a generalization of the previous models, called the \textit{Quermass Interaction Process}. In this model, the formal Hamiltonian is a linear combination of the fundamental Minkowski functionals, namely in $\Rd$ the area $\A$, the perimeter $\LL$ and the Euler-Poincaré characteristic $\x$:
$$
H = \theta_1 \A + \theta_2 \LL + \theta_3 \x ~.
$$
The existence of infinite volume Gibbs point processes for the Hamiltonian $H$ has been recently proved in \cite{Der09}. This paper focuses on the continuum percolation for such processes.\\

The existence of a percolation threshold $z^*$ for the Boolean model relies on a basic (but essential) monotonicity argument: see \cite{MR}, Chapter 2.2. This argument fails in the case of Gibbs point processes with Hamiltonian $H$. So, no percolation threshold can be expected in our context. However, other stochastic arguments as stochastic domination or FKG lead to percolation results. In \cite{CCK}, Chayes et al prove that percolation occurs for $z$ large enough and $\theta_2=\theta_3=0$. To our knowledge, the percolation phenomenon for other values of parameters $\theta_1,\theta_2,\theta_3$ has not been investigated yet.

Our main result (Theorem \ref{mainTH}) states that, for any $\theta_1,\theta_2,\theta_3$ (positive or negative), percolation occurs with probability $1$ for $z$ large enough. The only assumption bears on the random radii of the circular grains: they have to belong to a compact set not containing $0$. The proof of this theorem is relatively easy in the case $\theta_3=0$. Indeed, the \textit{local energy} $h((x,R),\o)$-- the energy variation when the grain $\bar B(x,R)$ is added to the configuration $\o$ -- is uniformly bounded (see Lemma \ref{Borneenergie}) and by classical stochastic comparison with respect to the Poisson Process the result follows. So the main challenge of the present paper is the proof of Theorem \ref{mainTH} when $\theta_3\neq 0$. In this setting, the local energy becomes unbounded from above and below and classical stochastic comparison arguments for point processes fail. In interpreting the percolation in our model via a site percolation model (see the beginning of Section \ref{SectionProofMainTH}), we prove the result thanks to a stochastic domination result for graphs due to  Liggett et al (Theorem 1.3 in \cite{LSS} or Lemma \ref{StochDomLiggett} below). An arduous control of the hole number variation, when a new grain is added, is the main technical issue. We prove essentially that this variation is moderate for a large enough set of admissible locations of grains. Let us mention that our proof is inspired by the one of Proposition 3.1 in \cite{GH}. 

Following \cite{CCK,GH}, we use our percolation result (Theorem \ref{mainTH}) to exhibit a phase transition phenomenon for Quermass model with several type of particles (Theorem \ref{theo:transition}).\\

Our paper is organized as follows. In Section \ref{QuermassModelsection}, the Quermass model and the main notations are introduced. The local energy $h((x,R),\o)$ is defined in (\ref{localenergy}). Section \ref{Resultssection} contains the results of the paper. Section \ref{stochasticargument} is devoted to the case $\theta_3=0$ and Section \ref{sectiontransition} to the phase transition result. The proof of Theorem \ref{mainTH} is developed in Section \ref{SectionProofMainTH}.

\section{Quermass Model}
\label{QuermassModelsection}

\subsection{Notations}

We denote by  $\B(\Rd)$ the set of bounded Borel sets in $\Rd$ with a positive Lebesgue measure. For any $\L$ and $\Delta$ in $\B(\Rd)$, $\L\oplus\Delta$ stands for the Minkoswki sum of these sets. Let $0\le R_0\le R_1$ be some positive reals and $\E$ be the product space $\Rd \times [R_0,R_1]$ endowed with its natural Euclidean Borel $\sigma$-algebra $\sigma(\E)$. For any $\L\in\B(\Rd)$, $\E_\L$ denotes the space $\L\times [R_0,R_1]$. A {\it configuration} $\o$ is a subset of $\E$ which is locally finite with respect to its first coordinate: $\#(\o\cap\E_\L)$ is finite for any $\L$ in $\B(\Rd)$. The configuration set $\O$ is endowed with the $\sigma$-algebra $\F$ generated by the functions $\o\mapsto\#(\o\cap A)$ for any $A$ in $\sigma(\E)$.\\
We will merely denote by $\o_\L$ instead of $\o\cap\E_\L$ the restriction of the configuration $\o$ (with respect to its first coordinate) to $\L$. Moreover, for any $(x,R)$ in $\E$, we will write $\o\cup (x,R)$ instead of $\o\cup\{(x,R)\}$.

A configuration $\o\in\O$ can be interpreted as a marked configuration on $\Rd$ with marks in $[R_0,R_1]$. To each $(x,R)\in\o$ is associated the closed ball $\bar B(x,R)$ (the grain) centered at $x$ (the germ) with radius $R$. The germ-grain surface $\bar\o$ is defined as
$$
\bar\o = \bigcup_{(x,R)\in\o} \bar B(x,R) ~.
$$

\subsection{Quermass interaction}

Let us define the Quermass interaction as in Kendall et al. \cite{KVB}. The energy (or Hamiltonian) of a finite configuration $\o$ in $\O$ is defined by
\begin{equation}
\label{energy}
H(\o)= \t_1 \A(\bar \o) + \t_2 \LL(\bar \o) +\t_3 \x(\bar\o) ~,
\end{equation}   
where $\t_1$, $\t_2$ and $\t_3$ are three real numbers, and $\A$, $\LL$ and $\x$ are the three fundamental Minkowski functionals, respectively area, perimeter and Euler-Poincaré characteristic. This last one is the difference between the number of connected components and the number of holes. Recall that a hole of $\bar\o$ is a bounded connected component of $\bar\o^c$. Hadwiger's Theorem ensures that any functional $F$ defined on the space of finite unions of convex compact sets, which is continuous for the Hausdorff topology, invariant under isometric transformations and additive (i.e. $F(A\cup B)=F(A)+F(B)-F(A\cap B)$) can be decomposed as in (\ref{energy}). This universal representation justifies the choice of the Quermass interaction for modelling mesoscopic random surfaces \cite{LMW,M}.   

The energy inside $\L\in\B(\Rd)$ of any given configuration $\o$ in $\O$ (finite or not) is defined by
\begin{equation}
\label{energyL}
H_\L(\o) = H(\o_{\Delta}) - H(\o_{\Delta\backslash \L}) ~,
\end{equation}
where $\Delta$ is any subset of $\Rd$ containing $\L\oplus B(0,2R_1)$. By additivity of functionals $\A$, $\LL$ and $\x$, the difference $H_\L(\o)$ does not depend on the chosen set $\Delta$.

Let us end with defining the local energy $h((x,R),\o)$ of the marked point $(x,R)\in\E$ (or of the associated ball $\bar B(x,R)$) with respect to the configuration $\o$:
\begin{equation}
\label{localenergy}
h((x,R),\o) = H_\L(\o\cup(x,R)) - H_\L(\o) ~,
\end{equation}
for any $\L\in\B(\Rd)$ containing $x$. Remark this definition does not depend on the choice of the set $\L$. The local energy $h((x,R),\o)$ represents the energy variation when the ball $\bar B(x,R)$ is added to the configuration $\o$.

\subsection{The Gibbs property}

Let $Q$ be a reference probability measure on $[R_0,R_1]$. Without loss of generality, $R_0$ and $R_1$ can be chosen such that, for every $\eps>0$, 
\begin{equation}
\label{optimalQ}
Q([R_0+\eps,R_1]) < 1 \; \text{ and } \; Q([R_0,R_1-\eps]) < 1 ~.
\end{equation}
Let $z>0$. Let us denote by $\l$ the Lebesgue measure on $\Rd$ and by $\pi^z$ the PPP on $\E$ with intensity measure $z\l\otimes Q$. Under $\pi^z$, the law of the random surface $\bar \o$ is the stationary boolean model with intensity $z>0$ and distribution of radius $Q$. Finally, for any $\L\in\B(\Rd)$, let us denote by $\pi^z_\L$ the PPP on $\E_\L$ with intensity measure $z\l_\L\otimes Q$, where $\l_\L$ is the restriction of the Lebesgue measure $\l$ to $\L$.

\begin{definition}
\label{Gibbs}
A probability measure $P$ on $\O$ is a Quermass Process for the intensity $z>0$ and the parameters $\theta_1, \theta_2, \theta_3$ if $P$ is stationary and if for every $\L$ in $\B(\Rd)$, for every bounded positive measurable function $f$ from $\O$ to $\R$, 
\begin{equation}
\label{DLR}
\int f(\o) P(d\o) =  \int\int f(\o'_\L\cup \o_{\L^c}) \frac{1}{Z_\L(\o_{\L^c})} e^{-H_\L(\o'_{\L}\cup\o_{\L^c})}\pi^z_\L(d\o'_\L) P(d\o) ~,
\end{equation}
where $Z_\L(\o_{\L^c})$ is the partition function
$$
Z_\L(\o_{\L^c}) = \int\int  e^{-H_\L(\o'_{\L}\cup\o_{\L^c})}\pi^z_\L(d\o'_\L) ~.
$$
\end{definition}

The equations (\ref{DLR})-- for all $\L\in\B(\Rd)$ --are called DLR for Dobrushin, Landford and Ruelle. They are equivalent to: for any $\L\in\B(\Rd)$, the law of $\o_\L$ under $P$ given $\o_{\L^c}$ is absolutely continuous with respect to the Poisson Process $\pi^z_\L$ with the local density
\begin{equation}
\label{localdensity}
g_\L(\o'_{\L}|\o_{\L^c}) = \frac{1}{Z_\L(\o_{\L^c})} e^{-H_\L(\o'_{\L}\cup\o_{\L^c})} ~.
\end{equation}
See \cite{P} for a general presentation of Gibbs measures and DLR equations.

The existence, the uniqueness or non-uniqueness (phase transition) of Quermass processes are difficult problems in statistical mechanics. The existence has been proved recently in \cite{Der09}, Theorem 2.1 for any parameters $z>0$ and $\theta_1, \theta_2, \theta_3$ in $\R$ . A phase transition result is proved in \cite{CCK,GLM,WR} for $R_0=R_1$, $\theta_2=\theta_3=0$ and for $\theta_1=z$ large enough.

\section{Results}
\label{Resultssection}

\subsection{Percolation occurs}

We say that \textit{percolation occurs} for a given configuration $\o\in\O$ if the subset $\bar\o$ of $\Rd$ contains at least one unbounded connected component. The set of configurations such that percolation occurs is a translation invariant event. Its probability, called \textit{the percolation probability}, equals to $0$ or $1$ for any ergodic Quermass process. However, the Quermass processes are not necessarily ergodic (they are only stationary) and their percolation probabilities may be different from $0$ and $1$. Besides, in \cite{CCK}, Chayes et al have built two Quermass processes, both corresponding to $\theta_2=\theta_3=0$ and $\theta_1=z$ large enough, whose percolation probabilities respectively equal to $0$ and $1$. Since any mixture of these two processes is still a Quermass process, the authors obtain Quermass processes whose percolation probabilities equal to any value between $0$ and $1$.\\
Our main result states that percolation occurs with probability $1$ for any (ergodic or not) Quermass process whenever the intensity $z$ is large enough.

\begin{theoreme}
\label{mainTH}
Let $R_0>0$ and $\theta_1,\theta_2,\theta_3\in\R$. There exists $z^{\ast}>0$ such that for any Quermass process $P$ associated to the parameters $\theta_1,\theta_2,\theta_3$ and $z>z^{\ast}$, percolation occurs $P$-almost surely.
\end{theoreme}

The proof of Theorem \ref{mainTH} is based on a discretization argument which allows to reduce the percolation problem from the (continuum) Quermass model to a site percolation model on the lattice $\Z^{2}$ (up to a scale factor). This proof is rather long and technical so it is addressed in Section \ref{SectionProofMainTH}.\\
Let us point out here that our theorem does not claim $z^{\ast}$ is a percolation threshold. In other words, for $z<z^{\ast}$, the percolation may be lost and recovered on different successive ranges.\\
Another natural question involves the number of unbounded connected components. Following the classical arguments for continuum percolation, we prove that this number is almost surely equal to zero or one.

\begin{proposition}
\label{numberInfCC}
For any Quermass process $P$ the number of unbounded connected component is a random variable in $\{0,1\}$.
\end{proposition}

\begin{proof}
It is well-known that any Gibbs measure is a mixture of extremal ergodic Gibbs measures. For each ergodic Quermass process $P$, the number of connected component is almost surely a constant in $\NN\cup\{+\infty\}$. For any $\L\in\B(\Rd)$, thanks to the DLR equations (\ref{DLR}), it is easy to prove that the law of $\o_\L$ under $P$ is equivalent to $\pi^{z}_\L$. Therefore, in following the general scheme of the proof of Theorem 2.1 in \cite{MR}, we show that the number of connected components is necessary $0$ or $1$. 
\end{proof}

\subsection{Percolation when $\theta_3=0$}
\label{stochasticargument}

In the particular case $\theta_3=0$, Theorem \ref{mainTH} can be completed and proved in a simple way.

First, let us recall the definitions involving the stochastic domination for point processes. We follow the notations given in \cite{GK}. An event $A$ in $\F$ is called increasing if for every $\o\in A$ and any $\o'\in\O$ containing $\o$  then $\o'\in A$ too. Let $P$ and $P'$ be two probability measures on $\O$. We say that $P$ is dominated by $P'$, denoted by $P\preceq P'$, if for every increasing event $A\in\F$, $P(A)\le P'(A)$. In this section, we focus our attention on the increasing event "there exists an unbounded connected component".

Let $P$ be any Quermass process and assume $\theta_3=0$ and $R_0>0$. Thanks to Lemma \ref{Borneenergie}, the local energy can be uniformly bound: there exist constants $C_0$ and $C_1$ such that for any $(x,R)\in\E$ and $\o\in\O$, 
\begin{equation}
\label{borneenergie}
C_0 \le h((x,R),\o) \le C_1 ~.
\end{equation}
Combining (\ref{borneenergie}) and Theorem 1.1 in \cite{GK}, we get the following stochastic dominations:
$$
\pi^{ze^{-C_1}} \preceq P \preceq  \pi^{ze^{-C_0}} ~.
$$
Now, the (stationary) Boolean models corresponding to $\pi^{ze^{-C_1}}$ and $\pi^{ze^{-C_0}}$ admit positive and finite percolation thresholds (see \cite{MRbook}, Chapter 3). It follows :

\begin{proposition}
\label{Propdomination}
Let $R_0>0$. For every $\theta_1, \theta_2$ in $\R$, there exist constants $z_0,z_1$ such that for any Quermas Process $P$ associated to parameters $z, \theta_1, \theta_2$ and $\theta_3=0$, the percolation occurs $P$-almost surely if $z>z_1$ and does not occur $P$-almost surely if $z<z_0$.  
\end{proposition}

Proposition \ref{Propdomination} improves Theorem \ref{mainTH} in the case $\theta_3=0$ since it ensures the existence of a subcritical regime.

It is worth pointing out here that the uniform bounds in (\ref{borneenergie}) do not hold whenever $\theta_3\neq 0$. Precisely, this is the hole number variation which cannot be uniformly bounded.

\subsection{Phase transition for multi-type Quermass Process}
\label{sectiontransition}

In this section, the multi-type Quermass model is introduced and a phase transition is exhibited, i.e. the existence of several Gibbs processes for the same parameters is proved.

Let $K$ be a positive integer. The $K$-type Quermass model is defined on the space $\Omega_K$ of configurations in $\E_K=\Rd \times [R_0,R_1]\times\{1,2,\ldots,K\}$. Each disc is now marked by a number specifying its type. We don't give the natural extension of the notations involving the sigma-field and so on.\\
The following Quermass energy function is defined such that all discs of a connected component have the same number. This is a non-overlapping multi-type germ-grain model. Precisely the energy of a finite configuration $\o$ is now given by 
\begin{equation}
\label{energymc}
H(\o)= \t_1 \A(\bar \o) + \t_2 \LL(\bar \o) +\t_3 \x(\bar\o)+ \sum_{
\begin{subarray}{c}
(x,R,i),(y,R',j)\in\o\\
i\neq j
\end{subarray}
} \phi(|x-y|-R-R') ~,
\end{equation}
where $\phi$ is an hardcore potential equals to infinity on $]-\infty,0]$ and zero on $]0,+\infty]$. The energy inside $\L\in\B(\Rd)$ of any finite or infinite configuration $\o$ is defined as in (\ref{energyL}) with the convention $+\infty-\infty=+\infty$. The definition of the $K$-type Quermass process via the DLR equations follows as in Definition \ref{Gibbs}.

The proof of the existence of such processes is similar to the one of the existence of Quermass process. See Theorem 2.1 of \cite{Der09} for more details. Here is our phase transition result:

\begin{theoreme}
\label{theo:transition}
Let $R_0>0$. For any $\theta_1$, $\theta_2$ and $\theta_3$ in $\R$, there exists $z_0>0$ such that, for any $z>z_0$, there exist several $K$-type Quermass Processes. The phase transition occurs. 
\end{theoreme}

The proof essentially follows the scheme of the one of Theorem 2.2 of \cite{CCK} or Theorem 1.1 of \cite{GH}. It is based on a random-cluster representation (or Gray Representation) analogous to the Fortuin-Kasteleyn representation of the Potts model. The existence of an unbounded connected component allows to prove the existence of a $K$-type Quermass process in which the density of particles of a given type is larger than the ones of the other types. By symmetry of the types, we prove the existence of at least $K$ different $K$-type Quermass processes.

\section{Proof of Theorem \ref{mainTH}}
\label{SectionProofMainTH}

\subsection{General scheme}

In the following, $P$ denotes a stationary Quermass process on $\O$ associated to the intensity $z>0$ and the parameters $\theta_1,\theta_2,\theta_3\in\R$.\\
Let $\ell$ be a real number such that $\ell>2R_{1}+2R_{0}$. Let us define the diamond box $\Delta$ as the interior of the convex hull of the eight points $(3\ell,0)$, $(6\ell,0)$, $(9\ell,3\ell)$, $(9\ell,6\ell)$, $(6\ell,9\ell)$, $(3\ell,9\ell)$, $(0,6\ell)$ and $(0,3\ell)$. This large octagon contains four smaller boxes $B_{\texttt{N}}$, $B_{\texttt{S}}$, $B_{\texttt{E}}$ and $B_{\texttt{W}}$ with side length $\ell$; precisely $B_{\texttt{N}}=(4\ell,7\ell)+[0,\ell]^{2}$, $B_{\texttt{S}}=(4\ell,\ell)+[0,\ell]^{2}$, $B_{\texttt{E}}=(7\ell,4\ell)+[0,\ell]^{2}$ and $B_{\texttt{W}}=(\ell,4\ell)+[0,\ell]^{2}$. The subscripts $\texttt{N}$, $\texttt{S}$, $\texttt{E}$ and $\texttt{W}$ refer to the cardinal directions. See Figure \ref{fig:diamondbox}. Thus, let us introduce the indicator function $\xi$ defined on $\O$ and equal to $1$ if and only if the two following conditions are satisfied:
\begin{itemize}
\item[(C1)] Each box $B_{\texttt{N}}$, $B_{\texttt{S}}$, $B_{\texttt{E}}$ and $B_{\texttt{W}}$, contains at least one point of ${\o}_{\Delta}$;
\item[(C2)] The number $N_{cc}^{\Delta}(\o)$ of connected components of $\bar{\o}_{\Delta}$ having at least one ball centered in one of the boxes $B_{\texttt{N}}$, $B_{\texttt{S}}$, $B_{\texttt{E}}$ or $B_{\texttt{W}}$, is equal to $1$.
\end{itemize}
In other words, $\xi(\o)=1$ means the boxes $B_{\texttt{N}}$, $B_{\texttt{S}}$, $B_{\texttt{E}}$ and $B_{\texttt{W}}$ are connected through $\bar{\o}_{\Delta}$.

\begin{figure}[!ht]
\begin{center}
\psfrag{a}{\small{$3\ell$}}
\psfrag{b}{\small{$6\ell$}}
\psfrag{c}{\small{$9\ell$}}
\includegraphics[width=7cm,height=7cm]{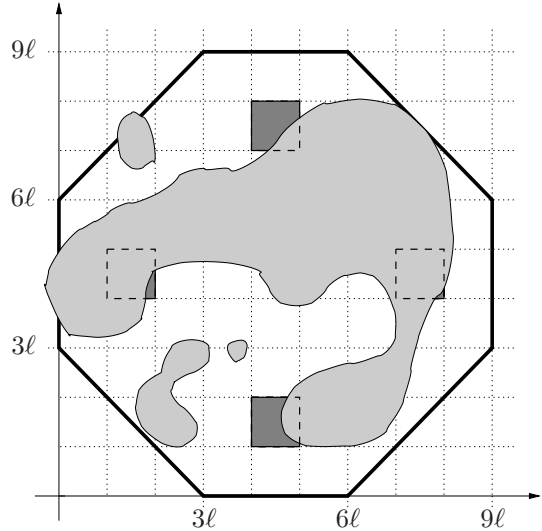}
\end{center}
\caption{\label{fig:diamondbox} \small{Here is the diamond box $\Delta$. The light gray set represents the configuration $\o$ restricted to $\Delta$. The dark gray squares are the fourth cardinal boxes $B_{\texttt{N}}$, $B_{\texttt{S}}$, $B_{\texttt{E}}$ and $B_{\texttt{W}}$ with side length $\ell$. On this picture, conditions (C1) and (C2) are fulfilled, i.e. $\xi(\o)=1$.}}
\end{figure}

For any $x\in(6\ell\,\Z)^{2}$, let $\tau_{x}$ be the translation operator on the configuration set $\E$ defined by $(y,R)\in\tau_{x}\o$ if and only if $(y+x,R)\in\o$. Hence, we can define the translated indicator function $\xi_{x}$ of $\xi$ on the translated box $\Delta_{x}=x+\Delta$ by $\xi_{x}(\o)=\xi(\tau_{x}\o)$. Let us remark that $\xi_{x}(\o)$ only depends on the restriction of the configuration $\o$ to the box $\Delta_{x}$. Moreover, thanks to the stationary character of the Quermass process $P$, the random variables $\xi_{x}$, $x\in(6\ell\,\Z)^{2}$, are identically distributed. They are dependent too.\\
Let us consider $x,y\in(6\ell\,\Z)^{2}$ such that $y=(6\ell,0)+x$. The boxes $\Delta_{x}$ and $\Delta_{y}$ have in common a cardinal box, i.e. $x+B_{\texttt{E}}=y+B_{\texttt{W}}$. So, the condition $\xi_{x}(\o)=\xi_{y}(\o)=1$ ensures that the cardinal boxes of $\Delta_{x}$ and $\Delta_{y}$ are connected together through the restriction of $\bar{\o}$ to $\Delta_{x}\cup\Delta_{y}$. The same is true when $y=(0,6\ell)+x$. This induces a graph structure on the vertex set $V=(6\ell\,\Z)^{2}$: for any $x,y\in V$, $\{x,y\}$ belongs to the edge set $E$ if and only if
$$
y-x \in \{ \pm(6\ell,0), \pm(0,6\ell) \} ~.
$$
The graph $(V,E)$ is merely the square lattice $\Z^{2}$ with the scale factor $6\ell$. The family $\{\xi_{x},x\in V\}$ provides a site percolation process on the graph $(V,E)$. It has been built so as to satisfy the following statement.

\begin{lemme}
\label{siteperco}
Let $\o\in\O$ such that percolation occurs in the site percolation process $\{\xi_{x},x\in V\}$. Then, so does for $\o$.
\end{lemme}

Let $\Pi_{p}$ be the Bernoulli (with parameter $p$) product measure on $\{0,1\}^{V}$. A stochastic domination result of Liggett et al \cite{LSS} (Theorem 1.3) allows to compare the site percolation processes induced by the family $\{\xi_{x},x\in V\}$ and $\Pi_{p}$. Here is an adaptated version to our context. Basic definitions about stochastic domination for lattice state spaces are not recall here. They are similar to the ones presented in Section \ref{stochasticargument} for point processes. See also \cite{Grim}.

\begin{lemme}
\label{StochDomLiggett}
Let $p\in[0,1]$. Assume that, for any vertex $x\in V$,
\begin{equation}
\label{UnifEspCond}
P \left( \xi_{x}=1 \,|\, \xi_{y} : \{x,y\}\notin E \right) \geq p \; \mbox{ a.s.}
\end{equation}
Then the distribution of the family $\{\xi_{x},x\in V\}$ stochastically dominates the probability measure $\Pi_{f(p)}$, where $f:[0,1]\to[0,1]$ is a deterministic function such that $f(p)$ tends to $1$ as $p$ tends to $1$.
\end{lemme}

Actually, Theorem \ref{mainTH} straight derives from Lemmas \ref{siteperco} and \ref{StochDomLiggett}. Let us first recall that in the site percolation model on the graph $(V,E)$, there exists a threshold value $p^{\ast}<1$ such that percolation occurs with $\Pi_{p}$-probability $1$ whenever $p>p^{\ast}$. See the book of Grimmett \cite{Grim}, p.~25. So, let $p$ be a real number in $[0,1]$ such that
\begin{equation}
\label{f(p)}
f(p) > p^{\ast} ~.
\end{equation}
Whenever the Quermass process $P$ satisfies (\ref{UnifEspCond}) for that $p$, then combining Lemmas \ref{siteperco} and \ref{StochDomLiggett} percolation occurs $P$-a.s. Therefore it remains to show that for any  $p>0$, hypothesis (\ref{UnifEspCond}) holds for $z$ large enough.\\

The next result claims that each Borel set of $\Rd$, sufficiently thick in some sense, contains at least one element of the configuration $\o$ with a probability tending to $1$ as the intensity $z$ tends to infinity. It will be proved at the end of this section.

\begin{lemme}
\label{LemmeNotempty}
Let $V\subset\Rd$ such that there exist $U\in\B(\Rd)$ with positive Lebesgue measure and $\eps>0$ satisfying $U\oplus\bar{B}(0,R_{1}+R_{0}+\eps)\subset V$. Then there exists a constant $C>0$, depending on $\l(U)$ and $\eps$, such that for any configuration $\o\in\O$ and for any $z>0$,
$$
P \left( \o_{V} = \emptyset \,|\, \o_{V^{c}} \right) \leq C z^{-1} ~.
$$
\end{lemme}

Since the Quermass process $P$ is stationary, it is sufficient to prove (\ref{UnifEspCond}) with $x=(0,0)$. So, we focus our attention on the diamond box $\Delta=\Delta_{(0,0)}$ and use Lemma \ref{LemmeNotempty} to check that condition (C1) is fulfilled in this box. Since $B_{\texttt{N}}$, $B_{\texttt{S}}$, $B_{\texttt{E}}$ and $B_{\texttt{W}}$ are sufficiently thick (with side length $\ell>2R_{1}+2R_{0}$), it follows
$$
P \left( \o_{B_{i}} = \emptyset \,|\, \o_{\Delta^{c}} \right) = P \left( P \left( \o_{B_{i}} = \emptyset \,|\, \o_{B_{i}^{c}} \right) \,|\, \o_{\Delta^{c}} \right) \leq C z^{-1} ~,
$$
for any $i\in\{\texttt{N},\texttt{S},\texttt{E},\texttt{W}\}$. So the conditional probability that $\o$ satisfies (C1) is larger than $1-4Cz^{-1}$.\\
The equation $N_{cc}^{\Delta}(\o)=0$ forces the box $B_{\texttt{N}}$ (for instance) to be empty of points of the configuration $\o$. Hence,
$$
P \left( N_{cc}^{\Delta}(\o) = 0 \,|\, \o_{\Delta^{c}} \right) \leq C z^{-1} ~.
$$
Checking that condition (C2) is fulfilled in the diamond box $\Delta$ needs what we call the Connection Lemma (Lemma \ref{ConnectionLemma}). This result states the conditional probability that $N_{cc}^{\Delta}(\o)$ is larger than $2$ converges to $0$ uniformly on the configuration outside $\Delta$. This is the heart of the proof of Theorem \ref{mainTH}. Its technical proof is given in Section \ref{SectionConnectLemma}.

\begin{lemme}[The Connection Lemma]
\label{ConnectionLemma}
There exists a constant $C'>0$ such that for any configuration $\o\in\O$ and for any $z>0$,
\begin{equation}
\label{Ncc>=2}
P \left( N_{cc}^{\Delta}(\o) \geq 2 \,|\, \o_{\Delta^{c}} \right) \leq C' z^{-1} ~.
\end{equation}
\end{lemme}

The above inequalities and the Connection Lemma imply that conditions (C1) and (C2) are fulfilled in $\Delta$ with a probability tending to $1$ as $z$ tends to $\infty$:
$$
P \left( \xi_{(0,0)}(\o) = 1 \,|\, \o_{\Delta^{c}} \right) \geq 1 - (5C+C') z^{-1} ~.
$$
The hypothesis (\ref{UnifEspCond}) then follows. Let $x$ be a vertex of the graph $(V,E)$ which is not a neighbor of $(0,0)$. By construction, the box $\Delta_{x}$ is included in $\Delta^{c}=\Delta_{(0,0)}^{c}$ (since $\Delta$ is an open set). This means the random variable $\xi_{x}$ is measurable with respect to the $\sigma$-algebra induced by the configurations restricted to $\Delta_{(0,0)}^{c}$. So,
$$
P \left( \xi_{(0,0)} = 1 \,|\, \xi_{x} : \{(0,0),x\}\notin E \right) \geq 1 - (5C+C') z^{-1} ~,
$$
and the hypothesis (\ref{UnifEspCond}) holds with $x=(0,0)$ and any $p\in[0,1[$, provided the intensity $z$ is large enough. This ends the proof of Theorem \ref{mainTH}.

\begin{proof}(Lemma \ref{LemmeNotempty})
Let $U\in\B(\Rd)$ be a bounded Borel set with positive Lebesgue measure and $V\supset U\oplus\bar{B}(0,R_{1}+R_{0}+\eps)$. First, let us write:
\begin{eqnarray}
\label{probavide}
P \left( \o_{V} = \emptyset \,|\, \o_{V^{c}} \right) & = & \frac{1}{Z_{V}(\o_{V^{c}})} \int_{\O_{V}} \1_{\o_{V}=\emptyset} \, e^{-H_{V}(\o_{V}\cup\o_{V^{c}})} \, \pi^{z}_{V}(d\o_{V}) \nonumber \\
& = & \frac{e^{-z\l(V)}}{Z_{V}(\o_{V^{c}})} ~,
\end{eqnarray}
since the empty configuration has a null energy, i.e. $H_{V}(\o_{V^{c}})=0$. A configuration $\o$ whose restriction to $V$ satisfies
$$
\#\o_{U\times[R_{0},R_{0}+\eps]}=1 \; \mbox{ and } \; \o_{V\setminus U}=\emptyset
$$
is reduced to a ball $\bar{B}(x,R)$ centered at a $x$ in $U$ and with a radius $R_{0}<R<R_{0}+\eps$. Since the ball $\bar{B}(x,R)$ does not overlap $\bar{\o}_{V^{c}}$, its energy $H_{V}((x,R)\cup\o_{V^{c}})$ is easy to compute;
$$
H_{V}((x,R)\cup\o_{V^{c}}) = \t_{1} 2\pi R + \t_{2} \pi R^{2} + \t_{3}
$$
(it is not worth using inequalities of Lemma \ref{Borneenergie} here). So, $H_{V}((x,R)\cup\o_{V^{c}})$ is bounded by a positive constant $K$ only depending on parameters $\t_{1}$, $\t_{2}$, $\t_{3}$ and radius $R_{1}$. Henceforth,
\begin{eqnarray*}
\lefteqn{P \left( \#\o_{U\times[R_{0},R_{0}+\eps]} = 1 , \; \o_{V\setminus U} = \emptyset \,|\, \o_{V^{c}} \right) }\hspace*{2cm} \\
& & = \frac{1}{Z_{V}(\o_{V^{c}})} \int_{U\times[R_{0},R_{0}+\eps]} e^{-H_{V}((x,R)\cup\o_{V^{c}})} \, z e^{-z\l(V)} \, \l(dx) Q(dR) \\
& & \geq \frac{e^{-z\l(V)}}{Z_{V}(\o_{V^{c}})} z e^{-K} \l(U) Q([R_{0},R_{0}+\eps]) ~.
\end{eqnarray*}
Recall that $Q([R_{0},R_{0}+\eps])$ is positive by (\ref{optimalQ}). 
Using the identity (\ref{probavide}), we finally upperbound the conditional probability $P(\o_{V}=\emptyset | \o_{V^{c}})$ by
$$
\left( z e^{-K} \l(U) Q([R_{0},R_{0}+\eps]) \right)^{-1} P \left( \#\o_{U\times[R_{0},R_{0}+\eps]} = 1 , \; \o_{V\setminus U} = \emptyset \,|\, \o_{V^{c}} \right) ~.
$$
This proves Lemma \ref{LemmeNotempty} with $C=(e^{-K} \l(U) Q([R_{0},R_{0}+\eps]))^{-1}$.
\end{proof}

\subsection{Proof of the Connection Lemma}
\label{SectionConnectLemma}

\subsubsection{Outline}

Let us recall that $N_{cc}^{\Delta}(\o)$ denotes the number of connected components of $\bar{\o}_{\Delta}$ having at least one ball centered in one of the four cardinal boxes $B_{\texttt{N}}$, $B_{\texttt{S}}$, $B_{\texttt{E}}$ or $B_{\texttt{W}}$. In this section, we assume $N_{cc}^{\Delta}(\o)\geq 2$. Our strategy consists in exhibiting a subset $B$ of the diamond box $\Delta$ in which $\o_{B}=\emptyset$. Moreover, for $x\in B$, if we are able to control uniformly the energy $H_{B}((x,R)\cup\o_{B^{c}})$ on $\o_{B^{c}}$, then the configuration $\o$ should contain a point centered in $B$ with large probability as $z$ tends to infinity. This leads to the Connection Lemma.

For $x\in B$, let us denote by $\mathcal{N}_{hol}((x,R),\o_{B^{c}})$ the hole number variation when the ball $\bar{B}(x,R)$ is added to the configuration $\o_{B^{c}}$. This quantity is central in our proof. Indeed, a first upperbound for the energy $H_{B}((x,R)\cup\o_{B^{c}})$ is given by Lemma \ref{Borneenergie}:
\begin{equation}
\label{firstUB}
H_{B}((x,R)\cup\o_{B^{c}}) = h((x,R),\o_{B^{c}}) \leq K - \t_{3} \, \mathcal{N}_{hol}((x,R),\o_{B^{c}}) ~,
\end{equation}
where $K$ is a positive constant only depending on parameters $\t_{1}$, $\t_{2}$, $\t_{3}$ and radii $R_{0}$, $R_{1}$. So, to upperbound the energy $H_{B}((x,R)\cup\o_{B^{c}})$ it suffices to upperbound the number of created holes (resp. deleted holes) when $\theta_{3}$ is negative (resp. positive). This is the reason why the proof of the Connection Lemma differs according to the sign of the parameter $\theta_{3}$.

\subsubsection{When $\theta_{3}$ is negative}
\label{section:theta3>0}

Let $\o$ be a configuration and $\alpha$ be a positive real number. A couple $(x,R)\in\R^{2}\times[R_{0},R_{1}]$ is said \textit{good} if all the connected components of the set $\bar{\o}_{\Delta}\cap\bar{B}(x,R)$ have an area larger than $\alpha$. These couples are well-named because adding a ball $\bar{B}(x,R)$ to the configuration $\o_{\Delta}$, with a good couple $(x,R)$, does not create too many holes.

\begin{lemme}
\label{BorneN_t}
Let $(x,R)\in\R^{2}\times[R_{0},R_{1}]$ be a good couple. Then,
$$
\mathcal{N}_{hol}((x,R),\o_{\Delta}) \leq \frac{\pi R_{1}^{2}}{\alpha} ~.
$$
\end{lemme}

\begin{proof}
The number of created holes when the ball $\bar{B}(x,R)$ is added to $\o_{\Delta}$ is smaller than the number of connected components of the set $\bar{\o}_{\Delta}\cap\bar{B}(x,R)$. Since $(x,R)$ is good, all these connected components have an area larger than $\alpha$. So, there are at most $\pi R^{2}/\alpha$ such connected components.
\end{proof}

Let us denote by $\textrm{Bad}(\o_{\Delta},\alpha)$ the following set:
$$
\textrm{Bad}(\o_{\Delta},\alpha) = \{ x \in \R^{2} , \; \exists R \in [R_{0},R_{0}+\eps] , \; (x,R) \; \mbox{ is not good }\} ~.
$$

\begin{lemme}
\label{limiteBad}
The area of the set $\textrm{Bad}(\o_{\Delta},\alpha)$ tends to $0$ as $\alpha$ and $\eps$ tend to $0$, uniformly on the configuration $\o_{\Delta}$.
\end{lemme}

Lemma \ref{limiteBad} will be proved at the end of this section. Thanks to Lemmas \ref{BorneN_t} and \ref{limiteBad}, we are now able to prove the Connection Lemma. First, we need a family of (small) non-overlapping squared boxes whose union covers the convex hull of the boxes $B_{\texttt{N}}$, $B_{\texttt{S}}$, $B_{\texttt{E}}$ and $B_{\texttt{W}}$. Precisely, for $\kappa>0$, let us consider a subset $\mathcal{B}$ of $\{v+[0,\kappa[^{2},v\in\R^{2}\}$ such that for any $B,B'$ in $\mathcal{B}$, $B\cap B'$ is empty, and
$$
\textrm{Conv}\left( B_{\texttt{N}}, B_{\texttt{S}}, B_{\texttt{E}}, B_{\texttt{W}} \right) \subset \bigcup_{B\in\mathcal{B}} B \subset \Delta ~.
$$
The family $\mathcal{B}$ is made up of at most $c_{\kappa}=\kappa^{-2}\A(\Delta)$ elements.\\
The hypothesis $N_{cc}^{\Delta}(\o)\geq 2$ ensures the existence of two elements $(x_{1},\cdot)$ and $(x_{2},\cdot)$ of $\o$, whose centers $x_{1}$ and $x_{2}$ are in the union of the four cardinal boxes $B_{\texttt{N}}$, $B_{\texttt{S}}$, $B_{\texttt{E}}$ and $B_{\texttt{W}}$, and whose balls $\bar{B}(x_{1},\cdot)$ and $\bar{B}(x_{2},\cdot)$ belong to two different connected components of $\bar{\o}$, say respectively $C_{1}$ and $C_{2}$. Let $[x_{1},x_{2}]$ be the segment in $\R^{2}$ linking $x_{1}$ with $x_{2}$ and $d$ be the euclidean distance on $\R^{2}$. The continuous application
$$
f : x \in [x_{1},x_{2}] \, \mapsto \, d(x,C_{1}) - d(x,\bar{\o}\setminus C_{1})
$$
satisfies $f(x_{1})<0$ and $f(x_{2})>0$. So there exists a point $x$ in $[x_{1},x_{2}]$ such that $d(x,C_{1})$ and $d(x,\bar{\o}\setminus C_{1})$ are equal (and positive). Hence, the ball $\bar{B}(x,R_{0})$ does not contain any point of ${\o}_{\Delta}$. Moreover, since $x$ is in the convex hull of the boxes $B_{\texttt{N}}$, $B_{\texttt{S}}$, $B_{\texttt{E}}$ and $B_{\texttt{W}}$, then it belongs to one box of the family $\mathcal{B}$, say $B$. With $\kappa<R_{0}/\sqrt{2}$, the box $B$ is contained in $\bar{B}(x,R_{0})$. Consequently, $\o_B$ is empty :
\begin{equation}
\label{choixB}
P \left( N_{cc}^{\Delta}(\o) \geq 2 \,|\, \o_{\Delta^{c}} \right) \leq \sum_{B\in\mathcal{B}} P \left( \o_{B} =  \emptyset \,|\, \o_{\Delta^{c}} \right) ~.
\end{equation}
For a given box $B\in\mathcal{B}$, let us consider the (random) set $U(\o_{\Delta\setminus B})$ of points $x\in B$ such that, for any radius $R\in [R_{0},R_{0}+\eps]$, the couple $(x,R)$ is good:
$$
U(\o_{\Delta\setminus B}) = B \setminus \textrm{Bad}(\o_{\Delta\setminus B},\alpha) ~.
$$
Let $x\in U(\o_{\Delta\setminus B})$ and $R\in [R_{0},R_{0}+\eps]$. On the one hand, using (\ref{firstUB}), $\t_{3}\leq 0$ and Lemma \ref{BorneN_t}, we get
\begin{equation}
\label{secondUBpositif}
H_{B}((x,R)\cup\o_{B^{c}}) \leq K - \t_{3} M ~,
\end{equation}
where $M=M(R_{1},\alpha)$ denotes the upperbound given by Lemma \ref{BorneN_t}. On the other hand, Lemma \ref{limiteBad} implies the area of $U(\o_{\Delta\setminus B})$ is larger than $\kappa^{2}/2$ for $\alpha$ and $\eps$ small enough, uniformly on the configuration $\o_{\Delta\setminus B}$. It follows:
\begin{eqnarray*}
\lefteqn{P \left( \#\o_{B\times[R_{0},R_{0}+\eps]} = 1 \,|\, \o_{B^{c}} \right) }\hspace*{2cm} \\
& & = \frac{1}{Z_{B}(\o_{B^{c}})} \int_{B\times[R_{0},R_{0}+\eps]} e^{-H_{B}((x,R)\cup\o_{B^{c}})} \, z e^{-z\l(B)} \, \l(dx) Q(dR) \\
& & \geq \frac{z e^{-z \kappa^{2}}}{Z_{B}(\o_{B^{c}})} \int_{U(\o_{\Delta\setminus B})\times[R_{0},R_{0}+\eps]} e^{-H_{B}((x,R)\cup\o_{B^{c}})} \, \l(dx) Q(dR) \\
& & \geq \frac{z e^{-z \kappa^{2}}}{Z_{B}(\o_{B^{c}})} e^{-K+\t_{3}M} \, \frac{\kappa^{2}}{2} \, Q([R_{0},R_{0}+\eps]) ~.
\end{eqnarray*}
In the previous inequality, replacing $e^{-z \kappa^{2}} Z_{B}(\o_{B^{c}})^{-1}$ with the conditional probability $P(\o_{B}=\emptyset | \o_{B^{c}})$, we obtain
$$
P \left( \o_{B} = \emptyset \,|\, \o_{B^{c}} \right) \leq \frac{2 \, e^{K-\t_{3}M}}{z \, \kappa^{2} \, Q([R_{0},R_{0}+\eps])} ~.
$$
Finally, the Connection Lemma derives from the above upperbound, (\ref{choixB}) and with
$$
C' = \frac{2 \, c_{\kappa} \, e^{K-\t_{3}M}}{\kappa^{2} \, Q([R_{0},R_{0}+\eps])} ~.
$$

\indent
In order to prove Lemma \ref{limiteBad}, we have to locate the set $\textrm{Bad}(\o_{\Delta},\alpha)$. Lemma \ref{LocalBad} says that the points $(x,.)$ in $\textrm{Bad}(\o_{\Delta},\alpha)$ are at distance around $R_{0}$ from $\bar{\o}_{\Delta}$. We need some notations. Let $\bar{B}(x,R)$ and $\bar{B}(y,\cdot)$ be two balls satisfying $R_{0}\leq R\leq R_{0}+\eps$ and
$$
0 < \A (\bar{B}(x,R)\cap \bar{B}(y,\cdot)) \leq \alpha ~.
$$
Then there exists a positive function $g(\eps,\alpha)$ tending to $0$ as $\alpha$ and $\eps$ tend to $0$, such that
\begin{equation}
\label{DistBad}
|\, d(x,\bar{B}(y,\cdot)) - R_{0} \,| \leq g(\eps,\alpha) ~.
\end{equation}
The function $g$ is also allowed to depend on radii $R_{0}$ and $R_{1}$. The topological boundary $\partial \bar{\o}_{\Delta}$ is composed of a finite number of arcs. Let $a$ be one of them. This arc is generated by an element of the configuration $\o_{\Delta}$, say $(y,\cdot)$. Now, we can define the circular strip $S_{g}(a)$ of width $2g(\eps,\alpha)$ by
$$
S_{g}(a) = \left\{ x \in \R^{2} ; \; \exists y' \in a \mbox{ s.t. }\begin{array}{c}
x = y + \mu (y'-y) \, \mbox{ with } \, \mu > 0 \, \mbox{ and} \\
|\, d(x,y') - R_{0} \,| \leq g(\eps,\alpha)
\end{array}
\right\} ~.
$$

\begin{lemme}
\label{LocalBad}
The following inclusion holds;
\begin{equation}
\label{InclusionBad}
\textrm{Bad}(\o_{\Delta},\alpha) \subset \bigcup_{a, \, \mbox{\footnotesize{arc of} } \partial \bar{\o}_{\Delta}} S_{g}(a) ~.
\end{equation}
\end{lemme}

\begin{proof}
Let us consider a point $x$ in $\textrm{Bad}(\o_{\Delta},\alpha)$. Let $R\in [R_{0},R_{0}+\eps]$ such that $(x,R)$ is not good. So there exists a connected component of $\bar{\o}_{\Delta}\cap\bar{B}(x,R)$ of area smaller than $\alpha$. The boundary of this connected component through the open ball $B(x,R)$ is composed of a finite number of arcs, say $a_{1},\ldots,a_{n}$. Let $a$ be one of them realizing the minima
$$
d(x,a) = \min_{1\leq i\leq n} d(x,a_{i}) ~.
$$
Let $(y,\cdot)$ be the element of the configuration $\o_{\Delta}$ generating the arc $a$. Let $S(a)$ be the semi-infinite cone centered at $y$ and with arc $a$ (i.e. the union of semi-line $[y,y')$ for $y'\in a$). Then, $x$ necessarily belongs to $S(a)$. Indeed, the opposite situation could lead to the existence of another arc $a'$ satisfying $d(x,a')<d(x,a)$. To sum up, $x$ is in the semi-infinite cone $S(a)$ and the area of $\bar{B}(x,R)\cap \bar{B}(y,\cdot)$ is positive and smaller than $\alpha$. So $x$ satisfies (\ref{DistBad}) and then belongs to $S_{g}(a)$.
\end{proof}

\begin{proof}[Proof of Lemma \ref{limiteBad}]
Let $a$ be an arc of the boundary $\partial \bar{\o}_{\Delta}$. Some geometrical considerations allow to bound the area of the circular strip $S_{g}(a)$:
$$
\A(S_{g}(a)) \leq 4 g(\eps,\alpha) \textrm{length}(a) ~,
$$
where $\textrm{length}(a)$ denotes the length of the arc $a$. We deduce from this bound and Lemmas \ref{LocalBad} and \ref{BornePerimetre}:
\begin{eqnarray*}
\A( \textrm{Bad}(\o_{\Delta},\alpha) ) & \leq & \sum_{a \, \mbox{\footnotesize{arc of} } \partial \bar{\o}_{\Delta}} \A( S_{g}(a) ) \\
& \leq & 4 g(\eps,\alpha) \sum_{a \, \mbox{\footnotesize{arc of} } \partial \bar{\o}_{\Delta}} \textrm{length}(a) \\
& \leq & 4 g(\eps,\alpha) \LL_{\Delta'}( \bar{\o}_{\Delta} ) \; \; \mbox{ with } \Delta'=\Delta\oplus B(0,R_{1}) \\
& \leq & 4 g(\eps,\alpha) \frac{2 \A( \Delta' \oplus B(0,R_{0}) )}{R_{0}} ~.
\end{eqnarray*}
This latter upperbound does not depend on the configuration $\o_{\Delta}$. So, this ends the proof of Lemma \ref{limiteBad}.
\end{proof}

\subsubsection{When $\theta_{3}$ is positive}

In this section, we still assume that $N_{cc}^{\Delta}(\o)$ is larger than $2$. But this time, our aim consists in upperbounding the number of deleted holes when the ball $\bar{B}(x,R)$, $x\in B$, is added to the configuration $\o_{B^{c}}$. The existence of a suitable set $B$ derives from Lemma \ref{ExistenceBoule}. Its proof is rather long and technical, mainly because of the uniformity of $\rho>0$ with respect to the configuration $\o$.

\begin{lemme}
\label{ExistenceBoule}
Assume $N_{cc}^{\Delta}(\o)\geq 2$. There exist $\rho>0$ (which does not depend on $\o$) and $O=O(\o)\in\Delta$ such that:
\begin{itemize}
\item[$(i)$] $O$ is in $\textrm{Conv}\left( B_{\texttt{N}}, B_{\texttt{S}}, B_{\texttt{E}}, B_{\texttt{W}} \right) \oplus B(0,\frac{3}{2} R_{0})$;
\item[$(ii)$] $B(O,\rho R_{0}) \cap \o$ is empty;
\item[$(iii)$] $B(O,(1+\rho) R_{0})$ does not (totally) contain any hole of $\bar{\o}$.
\end{itemize}
\end{lemme}

Let us first explain how the Connection Lemma straight derives from Lemma \ref{ExistenceBoule}. As in Section \ref{section:theta3>0}, we need the family $\mathcal{B}$ of non-overlapping squared boxes of length side $\kappa$. But here, $\mathcal{B}$ is required to cover a little bit more, i.e.
\begin{equation}
\label{B2}
\textrm{Conv}\left( B_{\texttt{N}}, B_{\texttt{S}}, B_{\texttt{E}}, B_{\texttt{W}} \right) \oplus B(0,\frac{3}{2} R_{0}) \subset \bigcup_{B\in\mathcal{B}} B ~,
\end{equation}
and parameters $\kappa$ and $\eps$ are chosen small enough so that
\begin{equation}
\label{KappaEpsRho}
\sqrt{2} \kappa + \eps < \rho R_{0}
\end{equation}
(where $\rho$ is given by Lemma \ref{ExistenceBoule}). Thanks to statement $(i)$ and (\ref{B2}), the point $O$ belongs to a box $B\in\mathcal{B}$. Thanks to $(ii)$, $(iii)$ and (\ref{KappaEpsRho}), $\o_{B}$ is empty and $\bar{\o}_{B^{c}}$ has no hole in $\mathbf{B}:=B\oplus B(0,R_{0}+\eps)$. Hence,
\begin{equation}
\label{choixB2}
P \left( N_{cc}^{\Delta}(\o) \geq 2 \,|\, \o_{\Delta^{c}} \right) \leq \sum_{B\in\mathcal{B}} P \left( P \left( \o_{B} =  \emptyset \,|\, \o_{B^{c}} \right) \1_{\bar{\o}_{B^{c}} \mbox{\footnotesize{ has no hole in }} \mathbf{B}} \,|\, \o_{\Delta^{c}} \right) ~.
\end{equation}
Let us pick a box $B\in\mathcal{B}$, a couple $(x,R)\in B\times[R_{0},R_{0}+\eps]$ and assume that $\bar{\o}_{B^{c}}$ has no hole in $\mathbf{B}$. Then, no hole is deleted when $\bar{B}(x,R)$ is added to $\o_{B^{c}}$. So, the hole number variation $\mathcal{N}_{hol}((x,R),\o_{B^{c}})$ is nonnegative. Combining with $\t_{3}\geq 0$ and (\ref{firstUB}), the energy $H_{B}((x,R)\cup\o_{B^{c}})$ is smaller than $K$ and we finish the proof of the Connection Lemma as in Section \ref{section:theta3>0}. First,
$$
P \left( \#\o_{B\times[R_{0},R_{0}+\eps]} = 1 \,|\, \o_{B^{c}} \right) \geq \frac{z e^{-z \kappa^{2}}}{Z_{B}(\o_{B^{c}})} e^{-K} \, \kappa^{2} \, Q([R_{0},R_{0}+\eps]) ~.
$$
Thus, replacing $e^{-z \kappa^{2}} Z_{B}(\o_{B^{c}})^{-1}$ by the conditional probability $P(\o_{B}=\emptyset | \o_{B^{c}})$, we get
$$
P \left( \o_{B} = \emptyset \,|\, \o_{B^{c}} \right) \leq \frac{e^{K}}{z \, \kappa^{2} \, Q([R_{0},R_{0}+\eps])} ~.
$$
Finally, the Connection Lemma derives from the above upperbound, (\ref{choixB2}) and with
$$
C' = \frac{c_{\kappa} \, e^{K}}{\kappa^{2} \, Q([R_{0},R_{0}+\eps])} ~,
$$
where $c_{\kappa}$ still denotes the number of boxes contained in the family $\mathcal{B}$.\\

Now, let us find a point $O$ and a radius $\rho>0$ satisfying the three properties of Lemma \ref{ExistenceBoule}. A first applicant for the point $O$ can be obtained following the same method as in Section \ref{section:theta3>0}. Based on the hypothesis $N_{cc}^{\Delta}(\o)\geq 2$, this method ensures the existence of a point $O'$ in the convex hull of the $B_{\texttt{N}}$, $B_{\texttt{S}}$, $B_{\texttt{E}}$, $B_{\texttt{W}}$'s, such that
\begin{equation}
\label{O'}
\mathrm{\mathbf{d}} := d(O',C_{1}) = d(O',\bar{\o}_{\Delta}\setminus C_{1}) >0
\end{equation}
where $C_{1}$ denotes a connected component of $\bar{\o}_{\Delta}$ counting by $N_{cc}^{\Delta}(\o)$. Two cases will be considered in the following. In the first one-- $\mathrm{\mathbf{d}}\geq \frac{1}{2} R_{0}$ --the connected components of $\bar{\o}_{\Delta}$ are far away from $O'$. So do their holes. Then, the choice $O=O'$ is appropriate. In the second case-- $\mathrm{\mathbf{d}}\leq\frac{1}{2} R_{0}$ --we exhibit a region close to $O'$ without hole and choose a suitable point $O$ inside. About the radius $\rho$, it will be proved in the sequel that any positive real number such that
\begin{equation}
\label{rhoC1}
\left( 1 + \rho \right)^{2} \, < \, 1 + \frac{1}{4} ~,
\end{equation}
\begin{equation}
\label{rhoC3}
\sqrt{7} ( 1+\rho ) - \frac{7}{4}\, < \, 1
\end{equation}
and
\begin{equation}
\label{rhoC2}
\left( 1-\frac{\sqrt{7}}{4}+\rho \right)^{2} + \left( \frac{3}{2}-\sqrt{\left( 1-\rho \right)^{2} - \left( 1-\frac{\sqrt{7}}{4}+\rho \right)^{2}} \right)^{2} \, < \, \left( \sqrt{3} - 1 - \rho \right)^{2} ~,
\end{equation}
is suitable. For instance, $\rho=0.01$ satisfies these three conditions.\\

\textbf{Case 1:} $\mathrm{\mathbf{d}}\geq \frac{1}{2} R_{0}$.\\
\noindent
By construction, $O'$ is in the convex hull of the boxes $B_{\texttt{N}}$, $B_{\texttt{S}}$, $B_{\texttt{E}}$, $B_{\texttt{W}}$ and is at distance at least $R_{0}+\mathrm{\mathbf{d}}$ from any point  $x$ in $\o_{\Delta}$. So, it satisfies properties $(i)$ and $(ii)$ of Lemma \ref{ExistenceBoule}. Now, let us consider a hole $T$ of $\bar{\o}_{\Delta}$. Assume in a first time that $O'$ does not belong to $T$. By (\ref{rhoC1}) and Lemma \ref{DistTrou1},
$$
d(O',T)^{2} \geq \left( 1+\frac{1}{4} \right) R_{0}^{2} \geq (1+\rho)^{2} R_{0}^{2} ~.
$$
This means that the hole $T$ is outside the ball $B(O',(1+\rho) R_{0})$. Now, assume that $O'$ is in $T$. Since $O'$ is equidistant from two connected components of $\bar{\o}_{\Delta}$ then one of them is inside the hole $T$. Hence, $T$ is too large to be totally covered by the ball $B(O',(1+\rho) R_{0})$. Consequently, $O'$ also satisfies $(iii)$.\\

\textbf{Case 2:} $\mathrm{\mathbf{d}}\leq\frac{1}{2} R_{0}$.\\
\noindent
Let $\bar{B}(x_{1},R_{x_{1}})$ be a ball of the connected component $C_{1}$ on which the distance $d(O',C_{1})$ is reached. Let us consider the point $y_{1}$ on the segment $[O',x_{1}]$ satisfying $\bar{B}(y_{1},R_{0})$ is included in $\bar{B}(x_{1},R_{x_{1}})$ and
$$
d(O',\bar{B}(y_{1},R_{0})) = d(O',\bar{B}(x_{1},R_{x_{1}})) = d(O',C_{1}) = \mathrm{\mathbf{d}} ~.
$$
In the same way, let us consider a point $y_{2}$ such that $\bar{B}(y_{2},R_{0})$ is included in $\bar{\o}_{\Delta}\setminus C_{1}$ and
$$
d(O',\bar{B}(y_{2},R_{0})) = d(O',\bar{\o}_{\Delta}\setminus C_{1}) = \mathrm{\mathbf{d}} ~.
$$
The region without hole, mentioned at the beginning of the current section and which we need, is built from points $y_{1}$ and $y_{2}$. See Figure \ref{fig:nohole}. Let $\mathcal{D}$ be the infinite line passing by $y_{1}$ and $y_{2}$. Thus, let us consider two infinite lines $\mathcal{D}'$ and $\mathcal{D}''$ parallel to $\mathcal{D}$ and such that
$$
d(\mathcal{D}',\mathcal{D}) = d(\mathcal{D}'',\mathcal{D}) = \frac{\sqrt{7}}{4} R_{0}
$$
(say $O'$ and $\mathcal{D}'$ are on the same side of the line $\mathcal{D}$). We denote by $\mathcal{H}$ the intersection of the convex hull of balls $\bar{B}(y_{1},R_{0})$ and $\bar{B}(y_{2},R_{0})$ with the strip delimited by $\mathcal{D}'$ and $\mathcal{D}''$. On Figure \ref{fig:nohole}, the border of $\mathcal{H}$ is drawn in bold.

\begin{lemme}
\label{noholeH}
With the previous notations and hypotheses, there is no hole in $\mathcal{H}$.
\end{lemme}

\begin{figure}[!ht]
\begin{center}
\psfrag{D1}{\small{$\mathcal{D}''$}}
\psfrag{D2}{\small{$\mathcal{D}$}}
\psfrag{D3}{\small{$\mathcal{D}_{O}$}}
\psfrag{D4}{\small{$\mathcal{D}'$}}
\psfrag{y1}{\small{$y_{1}$}}
\psfrag{y2}{\small{$y_{2}$}}
\psfrag{z1}{\small{$z_{1}$}}
\psfrag{z2}{\small{$z_{2}$}}
\psfrag{h1}{\small{$h_{1}$}}
\psfrag{M}{\small{$M$}}
\psfrag{O}{\small{$O$}}
\psfrag{O'}{\small{$O'$}}
\psfrag{H}{$\mathcal{H}$}
\psfrag{T}{$T$}
\psfrag{C1}{$C_{1}$}
\psfrag{C2}{$\bar{\o}_{\Delta}\setminus C_{1}$}
\includegraphics[width=12cm,height=6cm]{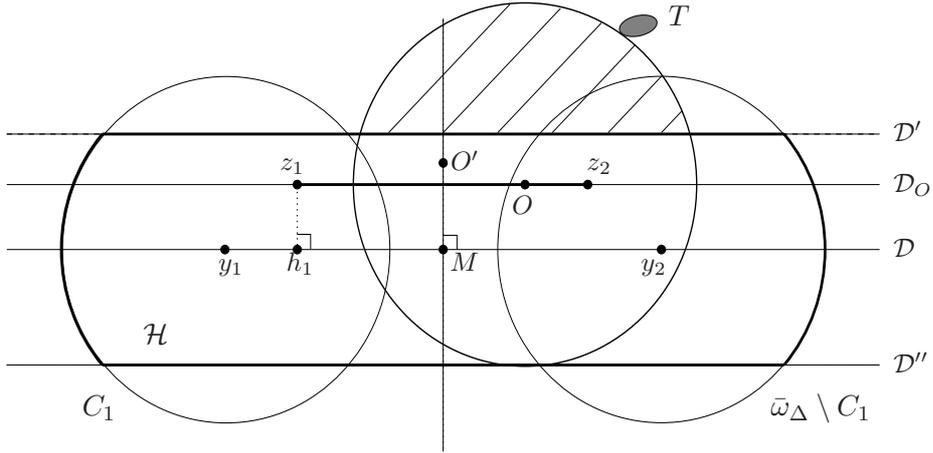}
\end{center}
\caption{\label{fig:nohole} \small{The balls $\bar{B}(y_{1},R_{0})$ and $\bar{B}(y_{2},R_{0})$ are respectively contained in the connected component $C_{1}$ and in $\bar{\o}_{\Delta}\setminus C_{1}$. From these balls a point $O$ is built and a real number $\rho>0$ is exhibited, satisfying together the three properties of Lemma \ref{ExistenceBoule}.}}
\end{figure}

\begin{proof}[Proof of Lemma \ref{noholeH}]
The closest hole $T$ to the segment $[y_{1},y_{2}]$ is obtained by pressing a ball with radius $R_{0}$ against $\bar{B}(y_{1},R_{0})$ and $\bar{B}(y_{2},R_{0})$. If $l$ denotes the distance between $T$ and $[y_{1},y_{2}]$ then $2l$ is the distance between the center of this pressing ball and $[y_{1},y_{2}]$. 
Pythagoras Theorem gives $(2l)^{2}+(R_{0}+h)^{2}=(2R_{0})^{2}$ in which $h$ denotes
$$
h := \frac{1}{2} d(y_{1},y_{2}) - R_{0} \, \leq \, \mathrm{\mathbf{d}} ~.
$$
In the worst case, $h=\frac{1}{2} R_{0}$. Hence, $l$ is always larger than $\frac{\sqrt{7}}{4} R_{0}$, which is the distance between $\mathcal{D}$ and $\mathcal{D}'$. To complete the proof, let us add there is no hole in the balls $\bar{B}(y_{1},R_{0})$ and $\bar{B}(y_{2},R_{0})$ since they are totally covered by $\bar{\o}_{\Delta}$.
\end{proof}

The idea to conclude the proof can be sum up as follows. The region $\mathcal{H}$ is sufficiently thick to contain strictly more than half of a ball with radius $(1+\rho)R_{0}$. Hence, the part of this ball outside $\mathcal{H}$ (this is the hatched region on Figure \ref{fig:nohole}) has a diameter smaller than $2R_{0}$. Thanks to Lemma \ref{DistTrou3}, it is possible to choose the center $O$ of this ball so that $\bar{B}(O,(1+\rho)R_{0})\cap \mathcal{H}^{c}$ does not contain any hole.\\
Let $\mathcal{D}_{O}$ be the infinite line parallel to $\mathcal{D}''$, at distance $(1+\rho)R_{0}$ from $\mathcal{D}''$ and on the same side as $\mathcal{D}$ of the line $\mathcal{D}''$. It derives from (\ref{rhoC3}) that the line $\mathcal{D}_{O}$ is trapped between $\mathcal{D}$ and $\mathcal{D}'$. Let $M$ be the center of the segment $[y_{1},y_{2}]$. Let us denote by $[z_{1},z_{2}]$ the following segment:
$$
[z_{1},z_{2}] := \bar{B}( M,(1-\rho)R_{0} ) \cap \mathcal{D}_{O} ~.
$$
See Figure \ref{fig:nohole}. We are going to choose the point $O$ on the segment $[z_{1},z_{2}]$. To do it, some geometrical results about the previous construction are needed. They will be proved at the end of the section:

\begin{lemme}
\label{3results}
With the previous notations and hypotheses, the following statements hold:
\begin{itemize}
\item[$(a)$] for $i=1,2$, $d(O',z_{i})\leq \frac{3}{2}R_{0}$;
\item[$(b)$] $[z_{1},z_{2}]\oplus B(0,\rho R_{0}) \subset B(O',R_{0}+\mathrm{\mathbf{d}})$;
\item[$(c)$] for $i=1,2$, $d(y_{i},z_{i}) \leq (\sqrt{3}-1-\rho) R_{0}$.
\end{itemize}
\end{lemme}

By convexity and statement $(a)$, any point of the segment $[z_{1},z_{2}]$ is at distance from $O'$ larger than $\frac{3}{2}R_{0}$. Moreover, $O'$ is in the convex hull of the $B_{\texttt{N}}$, $B_{\texttt{S}}$, $B_{\texttt{E}}$, $B_{\texttt{W}}$'s. Then, any point of $[z_{1},z_{2}]$ satisfies the property $(i)$ of Lemma \ref{ExistenceBoule}.\\
By construction of the point $O'$, the ball $B(O',R_{0}+\mathrm{\mathbf{d}})$ does not contain any point of $\o$. So does the set $[z_{1},z_{2}]\oplus B(0,\rho R_{0})$ thanks to statement $(b)$. This means that any point of the segment $[z_{1},z_{2}]$ satisfies the property $(ii)$ of Lemma \ref{ExistenceBoule}.\\
Combining statement $(c)$ with $i=1$ and Lemma \ref{DistTrou2}, we check there is no hole of $\bar{\o}_{\Delta}\setminus C_{1}$ in the ball $B(z_{1},(1+\rho) R_{0})$. Let us run the center of a ball with radius $(1+\rho) R_{0}$ along the segment $[z_{1},z_{2}]$ from $z_{1}$ to $z_{2}$ until that ball meets a hole of $\bar{\o}_{\Delta}\setminus C_{1}$. Two cases can be distinguished.
\begin{itemize}
\item[$\bullet$] This meet does not happen. Then, the ball $B(z_{2},(1+\rho) R_{0})$ does not contain any hole of $\bar{\o}_{\Delta}\setminus C_{1}$. It does not contain any hole of $C_{1}$ either thanks to statement $(c)$ with $i=2$ and Lemma \ref{DistTrou2}. In this case, $O=z_{2}$ satisfies the property $(iii)$ of Lemma \ref{ExistenceBoule}.
\item[$\bullet$] This meet happens: let $O$ be the corresponding center and $T$ be the corresponding hole of $\bar{\o}_{\Delta}\setminus C_{1}$. As just before, the ball $B(O,(1+\rho) R_{0})$ does not still contain any hole of $\bar{\o}_{\Delta}\setminus C_{1}$. Denote by $\mathcal{C}$ the part of this ball outside $\mathcal{H}$:
$$
\mathcal{C} := B(O,(1+\rho) R_{0}) \cap \mathcal{H}^{c} ~.
$$
On the one hand, the diameter of $\mathcal{C}$ is smaller than $2R_{0}$ thanks to (\ref{rhoC3}). On the other hand, $\mathcal{C}$ is pressed against the hole $T$ (there is no hole in $\mathcal{H}$); see Figure \ref{fig:nohole}. By Lemma \ref{DistTrou3}, the holes of the connected component $C_{1}$ are at distance from $T$ at least $2R_{0}$. So they cannot belong to the set $\mathcal{C}$. Therefore, this point $O$ satisfies the property $(iii)$ of Lemma \ref{ExistenceBoule}.
\end{itemize}

\begin{proof}[Proof of Lemma \ref{3results}]
The infinite line $\mathcal{D}$ divides $\bar{B}(M,R_{0})$ into two half-balls; let $\mathcal{V}$ be the one containing the segment $[z_{1},z_{2}]$. Since
$$
d(O',y_{1}) = d(O',y_{2}) = R_{0}+\mathrm{\mathbf{d}} ~,
$$
the half-ball $\mathcal{V}$ is included in the ball with center $O'$ and radius $R_{0}+\mathrm{\mathbf{d}}$. This inclusion admits two consequences. First, the points $z_{1}$ and $z_{2}$ which are in $\mathcal{V}$, are also in the ball $\bar{B}(O',R_{0}+\mathrm{\mathbf{d}})$. This implies, for $i=1,2$
$$
d(O',z_{i}) \leq \frac{3}{2} R_{0} ~,
$$
i.e. statement $(a)$. Second, the balls $\bar{B}(z_{i},\rho R_{0})$ which are included in $\mathcal{V}$, are also included in $\bar{B}(O',R_{0}+\mathrm{\mathbf{d}})$. So does the set $[z_{1},z_{2}]\oplus B(0,\rho R_{0})$ by convexity. Statement $(b)$ is proved. It remains to prove statement $(c)$. Let us introduce the orthogonal projection $h_{1}$ of $z_{1}$ over the infinite line $\mathcal{D}$ (see Figure \ref{fig:nohole}). Using $d(M,z_{1})=(1-\rho)R_{0}$, $d(h_{1},z_{1})=(1+\rho-\frac{\sqrt{7}}{4})R_{0}$ and $\mathrm{\mathbf{d}}\leq\frac{1}{2} R_{0}$, we get
$$
d(y_{1},z_{1}) \leq \sqrt{\left( 1-\frac{\sqrt{7}}{4}+\rho \right)^{2} + \left( \frac{3}{2}-\sqrt{\left( 1-\rho \right)^{2} - \left( 1-\frac{\sqrt{7}}{4}+\rho \right)^{2}} \right)^{2}} \, R_{0} ~.
$$
Thanks to (\ref{rhoC2}), statement $(c)$ follows.
\end{proof}

\subsection{Proofs of geometrical lemmas}
\label{SectionGeoLemma}


\begin{lemme}
\label{BornePerimetre}
Let $\Delta$ be a bounded closed convex set. For any configuration $\o$, let us denote by $\LL_\Delta(\bar\o)$ the perimeter of $\bar\o$ viewed through $\Delta$:
$$
\LL_\Delta(\bar\o) = \LL(\bar \o \cap \Delta) - \textrm{length}(\partial \Delta \cap \bar \o),
$$
where $\textrm{length}(\partial \Delta \cap \bar \o)$ denotes the lentgh of the boundary of $\Delta$ which is inside the set $\bar \o$. 
Then,
$$
\LL_\Delta(\bar\o) \leq \frac{2 \A(\Delta\oplus B(0,R_0))}{R_0} ~.
$$
\end{lemme}

\begin{proof}

The boundary of $\bar \o$ viewed through $\Delta$ corresponds to a finite union of arcs, say $(a_i)_{1\le i \le n}$. For each arc $a_i$, coming from the ball $B(x_i,R_i)$, we consider the circular strip $S(a_i)$ of width $R_{0}$ defined by
$$
S(a_i) = \left\{ x \in \R^{2} ; \; \exists x' \in a_i \mbox{ s.t. }\begin{array}{c}
x = x' + \mu (x_i-x') \, \mbox{ with } \, \mu > 0 \\
\mbox{ and } \; d(x,x')< R_{0} 
\end{array}
\right\} ~.
$$
Let us notice that the sets $(S(a_i))_{1\le i \le n}$ are disjoint. Indeed, let suppose that there exists $x\in S(a_i)\cap S(a_j)$ for some $i\neq j$. Without restriction, we can assume that the distance between $x$ and $a_i$ is smaller than or equal to the distance between $x$ and $a_j$. Let $y$ be the point on $a_i$ such that this distance is equal to $|y-x|$. Then, $y$ has to be strictly included in the ball $B(x_j,R_j)$ which contradicts the fact that $y$ is on the boundary of $\bar \o$.\\
This allows to compare the sum of the areas of $(S(a_i))_{1\le i \le n}$ with $\A(\bar\o)$:
\begin{eqnarray*}
\LL_\Delta(\bar\o) =  \sum_{i=1}^n \textrm{length}(a_i) & \leq & \frac{2}{R_0} \sum_{i=1}^n \S(a_i) \\
& \leq & \frac{2}{R_0} \A(\bar\o) \\
& \leq & \frac{2 \A(\Delta\oplus B(0,R_0))}{R_0} ~.
\end{eqnarray*}
\end{proof}

\begin{lemme}
\label{Borneenergie}
Let $\Delta$ be a bounded subset of $\R^{2}$, $\o$ be a configuration on $\Delta$ and $(x,R)$ be an element of $\Delta\times[R_0,R_1]$. Let us denote by $\A((x,R),\o)$ the area variation when the ball $\bar{B}(x,R)$ is adding to the configuration $\bar{\o}$:
$$
\A((x,R),\o) = \A((x,R)\cup\o) - \A(\o) ~.
$$
In the same way, we consider the perimeter variation $\LL((x,R),\o)$ and the connected component number variation $\mathcal{N}_{cc}((x,R),\o)$. The following inequalities hold.
\begin{equation}
\label{diffaire}
0 \leq \A((x,R),\o) \leq \pi R_{1}^{2} ~.
\end{equation}
\begin{equation}
\label{diffperimetre}
- \frac{2 \pi (R_{1}+R_{0})^{2}}{R_{0}} \leq \LL((x,R),\o) \leq 2 \pi R_{1} ~.
\end{equation}
\begin{equation}
\label{diffcc}
- \pi \left( 1 + \frac{R_{1}}{R_{0}} \right) \leq \mathcal{N}_{cc}((x,R),\o) \leq 1 ~.
\end{equation}
\end{lemme}

\begin{proof}
Inequalities (\ref{diffaire}), upperbounds of (\ref{diffperimetre}) and (\ref{diffcc}) are obvious. The border length of $\bar{\o}$ which is lost when the ball $\bar{B}(x,R)$ is adding can be interpreted as the perimeter of $\bar\o$ viewed through $\bar{B}(x,R)$ , i.e. as $\LL_{\bar{B}(x,R)}(\bar\o)$. Thanks to Lemma \ref{BornePerimetre}, it is smaller than
$$
\frac{2 \A(\bar{B}(x,R)\oplus B(0,R_0))}{R_0} \leq \frac{2 \pi (R_{1}+R_{0})^{2}}{R_{0}} ~.
$$
This gives the lowerbound of (\ref{diffperimetre}). It remains to lowerbound $\mathcal{N}_{cc}((x,R),\o)$. For that purpose, the number of deleted connected components when $\bar{B}(x,R)$ is adding to $\bar\o$, is smaller than the number of non-overlapping balls which overlap $\bar{B}(x,R)$. This number is at most
$$
\frac{2 \pi (R_{1}+R_{0})}{2 R_{0}} ~.
$$
\end{proof}

\begin{lemme}
\label{DistTrou1}
Let $\mathcal{C}$ be a connected component of $\bar{\o}_{\Delta}$ and $T$ be a hole of $\mathcal{C}$. Any point $x\in\R^{2}$ such that $x\notin\mathcal{C}$ and $x\notin T$ satisfies
$$
d(x,T)^{2} \geq d(x,\mathcal{C})^{2} + 2 d(x,\mathcal{C}) R_{0} ~.
$$
\end{lemme}

\begin{proof}

Let us consider a connected component $\mathcal{C}$, a hole $T$ and a point $x$ satisfying the assumptions of the lemma. Let $y$ be a point of the closure of $T$ such that $d(x,T)=|x-y|$. Necessarily, $y$ is on the boundary of two balls $B(z,R)$ and $B(z',R')$ of $\mathcal{C}$. Since $x$ belongs neither to $\mathcal{C}$ nor to $T$, at least one of $z-y$ or $z'-y$ has a nonnegative scalar product with $x-y$. Say $z-y$. Given $|x-z|$ and $|y-z|$, the distance $|x-y|$ is minimal when the vectors $z-y$ and $x-y$ are orthogonal. Hence, using $|x-z|\ge d(x,\mathcal{C})+R_0$ and $|y-z|\ge R_0$, it follows from Pythagoras Theorem that
$$
d(x,T)^2 \ge (d(x,\mathcal{C})+R_0)^2 - R_0^2 ~,
$$
which concludes the proof. 
\end{proof}

The following result is a straight consequence of Lemma \ref{DistTrou1}.

\begin{lemme}
\label{DistTrou2}
Let $\mathcal{C}$, $\mathcal{C}'$ be two connected components of $\bar{\o}_{\Delta}$. Let $\bar{B}(x,R)$ be a ball of $\mathcal{C}$ and $T'$ be a hole of $\mathcal{C}'$ which does not contain $\bar{B}(x,R)$. Then,
$$
d(x,T') \geq \sqrt{3} R_{0} ~.
$$
\end{lemme}

\begin{lemme}
\label{DistTrou3}
Let $T$ and $T'$ be two holes respectively of two connected components $\mathcal{C}$ and $\mathcal{C}'$ of $\bar{\o}_{\Delta}$. If $T\not\subset T'$ and $T'\not\subset T$ then
$$
d(T,T') \geq 2 R_{0} ~.
$$
\end{lemme}

\begin{proof}
Let $T$ and $T'$ be two holes satisfying the assumption of the lemma. We denote by $x$ and $y$ two points belonging respectively to the closure of $T$ and $T'$ such that $d(T,T')=|x-y|$. The point $x$ (respectively $y$) belongs to the boundary of two balls $B(z,R)$ and $B(z',R')$ of $\mathcal{C}$ (respectively $B(w,r)$ and $B(w',r')$ of $\mathcal{C}'$). An analysis, as in the proof of Lemma \ref{DistTrou1}, shows that the distance $|x-y|$ is minimal in the situation where $R=R'=r=r'=R_0$ and $z$, $z'$, $w$ and $w'$ form a parallelogram with length side $2R_0$. Then the points $x$ and $y$ are  at the middle of two opposite sides and the result follows.
\end{proof}

{\small \bibliographystyle{plain}
\bibliography{biblioGibbs}}

\end{document}